\documentclass[11pt,leqno]{amsart}

\usepackage[colorlinks=false,hidelinks]{hyperref} 
\usepackage[T1,T5]{fontenc} 
\usepackage[vietnamese,english]{babel}
\usepackage{mathtools}
\usepackage{amssymb}
\usepackage{amsrefs}
\usepackage[shortlabels]{enumitem}
\usepackage{graphicx}

\newtheorem{thm}{Theorem}[section]
\newtheorem{lemma}[thm]{Lemma}
\newtheorem{cor}[thm]{Corollary}
\newtheorem{prop}[thm]{Proposition}

\theoremstyle{definition}
\newtheorem{definition}[thm]{Definition}
\newtheorem{example}[thm]{Example}
\newtheorem{prob}[thm]{Problem}

\theoremstyle{remark}
\newtheorem{rmk}[thm]{Remark}
\newtheorem*{notation}{Notation}

\numberwithin{equation}{section}

\newcommand{\Z}{\mathbb{Z}}
\DeclareMathOperator{\id}{id}
\newcommand{\inv}{^{-1}}
\DeclareMathOperator{\RMult}{RMult}

\newcommand{\R}{\mathcal{R}}
\renewcommand{\L}{\mathcal{L}}
\newcommand{\F}{\mathcal{F}}
\renewcommand{\phi}{\varphi}
\newcommand{\bij}{\xrightarrow{\sim}}
\newcommand{\Grp}{\mathsf{Grp}}
\newcommand{\Ris}{\mathsf{Ris}}
\newcommand{\Risf}{\Ris_{\mathrm{fai}}}
\newcommand{\RQ}{\mathsf{RQuas}}
\newcommand{\tr}{\triangleright}

\hypersetup{
	pdflang={en-US},
	pdftitle={Groups versus quandle-like invariants of 3-manifolds},
	pdfauthor={Lực Ta},
	pdfsubject={Mathematics},
	pdfkeywords={3-manifold invariant, categorification, equivalence of categories, group, quandle, right quasigroup}
}

\begin{document}
	
	\title[
    Groups versus quandle-like invariants of 3-manifolds 
    ]{
    Groups versus quandle-like invariants of 3-manifolds
    }
	\author{L\d\uhorn c Ta}	
	
	\address{Department of Mathematics, University of Pittsburgh, Pittsburgh, Pennsylvania 15260}
	\email{ldt37@pitt.edu}
	
	
	
	\begin{abstract}
	Risandles are nonassociative algebraic structures recently introduced to construct invariants of $3$-manifolds. In this note, we show that the categories of groups and nonempty, faithful risandles are equivalent. In analogy to knot quandles, we also introduce \emph{fundamental risandles} of $3$-manifolds, which categorify the risandle coloring invariants of Ishii, Nakamura, and Saito \cite{risandle}. For infinitely many $3$-manifolds, the equivalence of categories recovers the fundamental group from the fundamental risandle and vice versa.
	\end{abstract}
	\maketitle
	
\section{Introduction}
In 2024, Ishii, Nakamura, and Saito \cite{risandle} introduced nonassociative algebraic structures called \emph{risandles} to construct coloring invariants of smooth, closed, connected, oriented 3-manifolds. The inspiration for risandles came from similar algebraic structures called \emph{quandles}, which Joyce \cite{joyce} and Matveev \cite{matveev} independently introduced in 1982 to construct complete knot invariants. 

Risandles are one of several kinds of quandle-like algebraic structures used to study invariants of 3-manifolds; see, for example, \citelist{\cite{fenn}\cite{nosaka}\cite{skew}\cite{hat}\cite{volume}\cite{dw}}. The classification problems for these structures are important for computing coloring invariants and cocycle invariants of knots and $3$-manifolds \cite{book}.

To address the classification problem for risandles, we prove the following. Let $\Grp$ and $\Risf$ denote the categories of groups and nonempty, faithful risandles, respectively.

\begin{thm}\label{thm:main}
    The functor $\R\colon \Grp\bij\Risf$ sending each group $G$ to the risandle $\R(G)\coloneq (G,\tr)$ with operation $h\tr g\coloneq hg\inv$ is an equivalence of categories. The inverse functor $\L\colon \Risf\bij\Grp$ sends each nonempty, faithful risandle $X$ to its right multiplication group $\RMult (X)$.
\end{thm}

In analogy with knot quandles, we also introduce \emph{fundamental risandles} $\F(M)$ of smooth, closed, connected, oriented $3$-manifolds $M$. Fundamental risandles categorify the risandle coloring invariants introduced in \cite{risandle} via \eqref{eq:cat}, and they also satisfy the following.

\begin{thm}[Theorem \ref{thm:invariant}]\label{thm:inv0}
    The fundamental risandle is an invariant of 3-manifolds up to orientation-preserving diffeomorphism. 
\end{thm}

\begin{thm}[Theorem \ref{thm:2}]\label{thm:3}
    There exist infinitely
    many 3-manifolds $M$ such that the fundamental risandle $\F(M)$ and the fundamental group $\pi_1(M)$ recover one another via $\L$ and $\R$.
    \end{thm}

\subsection{Structure of the paper}
In Section \ref{sec:prelims}, we recall several definitions and notational conventions from the theories of right quasigroups and quandles.

In Section \ref{sec:ris}, we define the category $\Ris$ of risandles, and we temporarily introduce algebraic structures called \emph{risacks} having one fewer axiom than risandles. We also discuss the functors $\L\dashv\R$ appearing in Theorem \ref{thm:main} as adjoint functors between $\Grp$ and $\Ris$.

In Section \ref{sec:props}, we record several algebraic properties of risandles. As applications, we show that all risacks are risandles and provide coordinate-free characterizations of risandles and faithful risandles.

In Section \ref{sec:pf}, we use the results of Section \ref{sec:props} to prove Theorem \ref{thm:main}.

In Section \ref{sec:free}, we discuss \emph{free risandles} $\langle X\rangle$ and quotients of risandles by congruence relations. These notions come from universal algebra, and we will need them to define fundamental risandles.

In Section \ref{sec:fund}, we introduce fundamental risandles $\F(M)$ of $3$-manifolds $M$ and prove Theorem \ref{thm:inv0} (Theorem \ref{thm:invariant}). We also explain how fundamental risandles categorify the risandle coloring invariants introduced in \cite{risandle}; see \eqref{eq:cat}.

In Section \ref{sec:lens}, we study the fundamental risandle $\F(M)$ of the lens space $M=L(n,1)$, where $n\in\Z^+$ is a positive integer.

In Section \ref{sec:end}, we pose two open questions about the relationship between fundamental risandles $\F(M)$ and fundamental groups $\pi_1(M)$. Finally, we prove Theorem \ref{thm:3} (Theorem \ref{thm:2}).

\begin{notation}
    We use the following notation throughout this paper. Denote the composition of functions $\phi\colon X\to Y$ and $\psi\colon Y\to Z$ by $\psi\phi$. Given a set $X$, let $S_X$ denote the symmetric group of $X$, and let $\id_X\in S_X$ denote the identity map of $X$. Let $\Z^+$ denote the set of positive integers.
\end{notation}

\section{Right quasigroups}\label{sec:prelims}

Following \citelist{\cite{taGQ}\cite{bonatto}\cite{quasi}} in analogy with quandles, we will define risandles as a special class of \emph{right quasigroups} (that is, magmas such that all right multiplication maps are permutations).

\begin{definition}[\citelist{\cite{bonatto}\cite{quasi}}]\label{def:rq}
     A \emph{right quasigroup} is a pair $(X,s)$ where $X$ is a set and $s\colon X\to S_X$ is a function from $X$ to its symmetric group. For all $x\in X$, we call the map
     \[
     s_x\coloneq s(x)
     \]
     a \emph{right multiplication map} or \emph{right translation}. If $s_x=\id_X$, then we call $x$ a \emph{right identity element} of $X$.  
     Finally, if $s$ is injective, then we say that $(X,s)$ is \emph{faithful}. 
\end{definition}

Equivalently, a right quasigroup is a pair $(X,\tr_X)$ where $X$ is a set and $\tr_X\colon X\times X\to X$ is a binary operation such that, for all $x\in X$, the function $y\mapsto y\tr_X x$ is a permutation of $X$. We will omit the subscript from $\tr_X$ when it is clear from the context.\footnote{This paper borrows the notation $s_x(y)$ and $\tr$ from quandle theory to hint at the analogies between quandles and risandles; cf.\ \cite{risandle}. In particular, neither quandle operations nor risandle operations $\tr$ are associative in general.}

The equivalence of this definition with Definition \ref{def:rq} is given by the formula
\[
s_x(y)=y\tr x.
\]
Each of these two definitions has advantages over the other; this paper uses them interchangeably.

\begin{definition}[\citelist{\cite{bonatto}\cite{quasi}}]
    The \emph{right multiplication group} of a right quasigroup $(X,s)$, denoted by $\RMult(X)$, is the subgroup of $S_X$ generated by the set $s(X)$ of right multiplication maps:
    \[
    \RMult(X)\coloneq \langle s_x\mid x\in X\rangle\leq S_X.
    \]
\end{definition}

\begin{definition}[\citelist{\cite{bonatto}\cite{quasi}}]
    A \emph{homomorphism} of right quasigroups, say from $(X,\tr_X)$ to $(Y,\tr_Y)$, is a function $\phi\colon X\to Y$ that preserves the binary operations of $X$ and $Y$; that is,
    \[
    \phi(w\tr_X x)=\phi(w)\tr_Y \phi(x)
    \]
    for all $w,x\in X$. Let $\RQ$ denote the category of right quasigroups and their homomorphisms. Then an \emph{isomorphism} in $\RQ$ is simply a bijective homomorphism.
\end{definition}

\begin{example}
    The (right) regular action of a group $G$ on itself, given by $h\tr g\coloneq hg$, defines a right quasigroup $(G,\tr)$. Denote the category of groups by $\Grp$. Under this view, it is easy to see that $\Grp$ embeds faithfully but not fully into $\RQ$.
\end{example}

\begin{example}[\citelist{\cite{fenn}\cite{joyce}\cite{matveev}\cite{quasi}}]
    A \emph{rack} is a right quasigroup such that all right multiplication maps $s_x$ are endomorphisms. A \emph{quandle} is a rack $(X,\tr)$ such that $x\tr x=x$ for all $x\in X$. The right multiplication group of a rack is often called its \emph{inner automorphism group} or \emph{operator group}.
\end{example}

\section{Risandles}\label{sec:ris}

In this section, we study the algebraic properties of \emph{risandles}, which Ishii, Nakamura, and Saito introduced in  \cite{risandle} with inspiration from quandle theory. 

\subsection{Risacks and risandles}
Analogizing the relationship between racks and quandles (see \cite{fenn}), we temporarily introduce \emph{risacks}, which have one fewer axiom than risandles. Like racks and risandles, risacks have a topological motivation; see \cite{risandle}*{Rem.\ 6.7}. However, we later show that risacks are the same as risandles (Proposition \ref{prop:red}).

\begin{definition}[Cf.\ \cite{risandle}]
    Let $(X,\tr)$ be a right quasigroup. We say that $(X,\tr)$ is a \emph{risack} if
    \begin{equation}\label{eq:axiom}
        y\tr z=(y\tr x)\tr(z\tr x)
    \end{equation}
    for all $x,y,z\in X$. Equivalently, \begin{equation}\label{eq:axiom2}
        s_z=s_{s_x(z)}s_x
    \end{equation}
    for all $x,z\in X$.
\end{definition}

\begin{definition}[\cite{risandle}]\label{def}
    Let $(X,\tr)$ be a risack. We say that $(X,\tr)$ is a \emph{risandle} if
    \begin{equation}\label{eq:red}
        x=(x\tr x)\tr((x\tr x)\tr x)
    \end{equation}
    for all $x\in X$. 
    Let $\Ris$ be the full subcategory of $\RQ$ whose objects are risandles; a \emph{risandle homomorphism} (resp.\ \emph{isomorphism}) is just a morphism (resp.\ bijective morphism) in $\Ris$. Let $\Risf$ be the full subcategory of $\Ris$ whose objects are nonempty and faithful.
\end{definition}

\begin{rmk}
    In Proposition \ref{prop:red}, we show that all risacks are risandles, making axiom \eqref{eq:red} redundant. Moreover, Corollary \ref{cor:alt} provides a coordinate-free characterization of risandles.
\end{rmk}

\begin{example}[\cite{risandle}*{Prop.\ 5.1}]\label{ex:r}
    Let $G$ be a group, and define a binary operation $\tr\colon G\times G\to G$ by
    \[
    h\tr g\coloneq hg\inv.
    \]
    Then the pair $\R(G)\coloneq (G,\tr)$ is a faithful risandle.
\end{example}

\begin{example}[Cf.\ \cite{book}]
    There exist infinitely many unfaithful risandles. For example, let $X$ be any set, and define $s_x\coloneq \id_X$ for all $x\in X$. Then $(X,s)$ is a quandle called a \emph{trivial quandle}. Clearly, $(X,s)$ is also a risandle. (In fact, a converse to this statement holds; see Proposition \ref{prop:rack}.) If $|X|\geq 2$, then $(X,s)$ is unfaithful. Moreover, for all nontrivial risandles $(Y,t)$, the \emph{product risandle} $(X\times Y,s\times t)$ (that is, the categorical product in $\Ris$) is nontrivial and unfaithful.
\end{example}

\subsection{The functors $\R$ and $\L$}
It is easy to see that the assignment $\R$ from Example \ref{ex:r} defines a faithful functor $\R\colon \Grp\to \Risf$. Expanding the codomain to $\Ris$, we observe that $\R\colon \Grp\to\Ris$ has a left adjoint $\L\colon \Ris\to\Grp$ that sends a risandle $(X,\tr)$ to the group
\[
\L(X)\coloneq \langle e_x\ (x\in X)\mid e_{x\tr y}=e_x e_y\inv \rangle.
\]

\begin{rmk}
    The adjunction $\L\dashv\R$ may be viewed as a risandle-theoretic analogue of the adjunction $\operatorname{Adconj}\dashv\operatorname{Conj}$ between quandles and groups introduced in \citelist{\cite{joyce}\cite{matveev}}; cf.\ \cite{book}. 
    
    However, while $\operatorname{Adconj}$ always sends nonempty quandles to infinite groups (see \cite{book}*{Lem.\ 2.27}), Theorem \ref{thm:main} shows that $|\L(X)|=|X|$ for all nonempty, faithful risandles $(X,s)$. In fact, $\L(X)$ may be finite even if $(X,s)$ is unfaithful. For example, if $(X,s)$ is a trivial quandle, then $\L(X)$ is the trivial group.
\end{rmk}

\section{Properties of risandles}\label{sec:props}
In this section, we discuss several algebraic properties of risandles. 

\subsection{Preliminary results}
First, we give an analogue for risacks of the following fact proven in \cite{taGQ}*{Prop.\ 2.11}: \emph{Let $(X,s)$ be a right quasigroup. Then $(X,s)$ is a rack if and only if $s\colon X\to S_{X}$ is a homomorphism into the conjugation quandle $\operatorname{Conj}(S_{X})$.}

\begin{lemma}\label{lem:alt}
    Let $(X,s)$ be a right quasigroup. Then $(X,s)$ is a risack if and only if $s\colon X\to S_X$ is a homomorphism into the risandle $\R(S_X)$.
\end{lemma}

\begin{proof}
    ``$\implies$'' Suppose that $(X,s)$ is a risack. For all $x,z\in X$, the risack axiom \eqref{eq:axiom2} implies that
    \[
    s(z\tr_X x)=s(s_x(z))=s_{s_x(z)}=s_zs_x\inv =s_z \tr_{S_X}s_x=s(z)\tr_{S_X}s(x),
    \]
    so $s$ is a homomorphism.

    ``$\impliedby$'' The proof is similar.
\end{proof}

\begin{rmk}
    Lemma \ref{lem:alt} can be strengthened; see Corollary \ref{cor:alt} and cf.\ Proposition \ref{prop:bij}.
\end{rmk}

Next, we study right identity elements of risacks.

\begin{prop}\label{prop:id1}
    Let $(X,s)$ be a risack.
    Then for all $x\in X$, the element $x\tr x$ is a right identity element of $X$.
\end{prop}

\begin{proof}
    We have to show that $a=a\tr(x\tr x)$ for all $x,a\in X$. To that end, define the element \[
    y\coloneq s_x\inv(a).
    \] Then
    \[
    a=s_x(y)=y\tr x=(y\tr x)\tr(x\tr x)=a\tr (x\tr x),
    \]
    where in the third equality we have used the risack axiom \eqref{eq:axiom} with $z\coloneq x$.
\end{proof}

\begin{cor}\label{cor:unique}
    If $(X,s)$ is a nonempty, faithful risack, then $X$ contains a unique right identity element $e$. Moreover, $e=x\tr x$ for all $x\in X$.
\end{cor}

\begin{rmk}
    Corollary \ref{cor:unique} can also be deduced from Theorem \ref{thm:main}.
\end{rmk}

\begin{prop}\label{prop:rack}
    Let $(X,s)$ be a rack. Then $(X,s)$ is a risack if and only if it is a trivial quandle.
\end{prop}

\begin{proof}
    ``$\implies$''
    If $(X,s)$ is both a rack and a risack, then it is easy to see from the risack axiom \eqref{eq:axiom2} that $s_z=s_xs_z$ for all $x,z\in X$. Hence, $s_x=\id_X$ for all $x\in X$, as desired.

    ``$\impliedby$'' Clear.
\end{proof}

\subsection{Alternative definitions of risandles} We provide two characterizations of risandles that are equivalent to Definition \ref{def}. In the following, let $(X,s)$ be a right quasigroup.
\begin{prop}\label{prop:red}
    $(X,s)$ is a risack if and only if it is a risandle. Thus, the risandle axiom \eqref{eq:red} is redundant.
\end{prop}

\begin{proof}
    We have to verify the risandle axiom \eqref{eq:red} for all elements $x\in X$. To that end, define the element $z\coloneq x\tr x$. Then the risack axiom \eqref{eq:axiom2} yields
    \begin{equation}\label{eq:inve}
        s_x\inv=s_z\inv s_{s_x(z)}=\id_X\inv s_{s_x(z)}=s_{s_x(z)},
    \end{equation}
    where in the second equality we have used Proposition \ref{prop:id1}.
    Hence,
    \[
    x=s_x\inv s_x(x)=s_x\inv(z)=s_{s_x(z)}(z)=z\tr(z\tr x)=(x\tr x)\tr ((x\tr x)\tr x),
    \]
    as desired.
\end{proof}

\begin{rmk}
    In \cite{risandle}*{Rem.\ 6.7}, Ishii, Nakamura, and Saito speculated that if there exist risacks that are not risandles, then those risacks could produce homotopy invariants of non-singular flows of a given $3$-manifold. However, Proposition \ref{prop:red} states that no such risacks exist.
\end{rmk}

\begin{rmk}
    While Proposition \ref{prop:red} shows that the risack axiom \eqref{eq:axiom} implies the risandle axiom \eqref{eq:red}, the converse is not true. For example, let $X$ be any set containing at least two elements, let $\sigma\in S_X$ be any nonidentity involution of $X$, and define $s_x\coloneq \sigma$ for all $x\in X$. Then $(X,s)$ is a right quasigroup satisfying the risandle axiom \eqref{eq:red}. At the same time, $(X,s)$ is a nontrivial \emph{permutation rack} or \emph{constant action rack} (cf.\ \cite{fenn}*{Ex.\ 7 in Sec.\ 1}), so by Proposition \ref{prop:rack}, it is not a risack.
\end{rmk}

Proposition \ref{prop:red} states that the definition of risacks coincides with that of risandles. 
Combining this fact with Lemma \ref{lem:alt} yields a coordinate-free characterization of risandles.

\begin{cor}\label{cor:alt}
    $(X,s)$ is a risandle if and only if $s$ is a homomorphism into the risandle $\R(S_X)$.
\end{cor}

\subsection{Faithful risandles}
We strengthen Corollary \ref{cor:alt} in the case that $(X,s)$ is faithful.

\begin{lemma}\label{lem:sets}
    For all nonempty risandles $(X,s)$, we have an equality of sets
    \[
    \RMult(X)=s(X)\subseteq S_X.
    \]
    In other words, $s$ is a surjection onto $\RMult(X)$.
\end{lemma}

\begin{proof}
    By the definition of $\RMult(X)$, it suffices to show that $s(X)$ is a subgroup of $S_X$. Existence of identity and closure under inversion follow from Proposition \ref{prop:id1} and \eqref{eq:inve}, respectively. Given $s_x,s_y\in s(X)$, apply the risack axiom \eqref{eq:axiom2} with the element $z\coloneq s_x\inv(y)$ to obtain
    \[
    s_ys_x=s_{s_x(z)}s_x=s_z\in s(X),
    \]
    showing that $s(X)$ is closed under the group operation of $S_X$.
\end{proof}

\begin{prop}\label{prop:bij}
    Let $(X,s)$ be a nonempty risandle. Then $(X,s)$ is faithful if and only if $s$ is a bijection onto $\RMult(X)$. In this case, $s$ is actually a risandle isomorphism
    \[
    s\colon (X,s)\bij \R(\RMult(X)).
    \]
\end{prop}

\begin{proof}
    This follows directly from Corollary \ref{cor:alt} and Lemma \ref{lem:sets}.
\end{proof}

\begin{cor}
    Every faithful risandle $(X,s)$ is a \emph{quasigroup}. That is, $(X,s)$ is a right quasigroup such that for all $y\in X$, the \emph{left multiplication map} $x\mapsto y\tr x$ is a permutation of $X$ (cf.\ \cite{quasi}).
\end{cor}

\begin{proof}
    It is easy to see that for all groups $G$, the risandle $\R(G)$ is a quasigroup. Hence, the claim follows from Proposition \ref{prop:bij}.
\end{proof}

\section{Proof of Theorem \ref{thm:main}}\label{sec:pf}

In this section, we prove our main result. Given a group $G$, denote its identity element by $1_G$.

\begin{prop}\label{prop:ff}
    $\R$ is fully faithful.
\end{prop}

\begin{proof}
    Clearly, $\R$ is faithful. To show that $\R$ is full, let $G$ and $K$ be groups, and let $\phi\colon \R(G)\to \R(K)$ be a risandle homomorphism. We have to show that $\phi$ is a group homomorphism when considered as a set-theoretic map $\phi\colon G\to K$. First, $\phi$ preserves identity elements because
    \[
    \phi(1_G)=\phi(1_G\tr_G 1_G)=\phi(1_G)\tr_K\phi(1_G)=\phi(1_G)\phi(1_G)\inv=1_K,
    \]
    as desired. For all $g,h\in G$, we have
    \[
    \phi(gh)=\phi(g(h\inv)\inv)=\phi(g\tr_G(1_G\tr_G h))=\phi(g)\tr_K(\phi(1_G)\tr_K\phi(h)),
    \]
    so by what we just showed,
    \[
    \phi(gh)=\phi(g)\tr_K(1_K\tr_K\phi(h))=\phi(g)(\phi(h)\inv)\inv=\phi(g)\phi(h),
    \]
    as desired.
\end{proof}

\begin{proof}[Proof of Theorem \ref{thm:main}]
    By Proposition \ref{prop:bij}, $\R\colon \Grp\to\Risf$ is essentially surjective, so by Proposition \ref{prop:ff}, $\R$ is an equivalence of categories. Hence, Proposition \ref{prop:bij} provides natural isomorphisms
    \[
    \L(X)\cong \L\R(\RMult(X))\cong\RMult(X)
    \]
    for all faithful risandles $(X,s)$.
\end{proof}

\begin{cor}
    Every group is isomorphic to the right multiplication group of a faithful risandle.
\end{cor}

\begin{proof}
    For all groups $G$, Theorem \ref{thm:main} implies that $G\cong \RMult(\R(G))$.
\end{proof}

\begin{rmk}
    The isomorphism $G\cong \RMult(\R(G))$ can be verified directly without any difficulty.
\end{rmk}

\section{Free risandles}\label{sec:free}
In this section, we discuss \emph{free risandles} and quotients of risandles by \emph{congruence relations}. These are special cases of notions coming from universal algebra; we refer the reader to \cite{alg-theories} for a reference. 

By Proposition \ref{prop:red}, risandles may be viewed as an algebraic theory with two binary operations $s_-(-)$ and $s_-\inv(-)$ that satisfy the equational laws \begin{equation}\label{eq:rlns}
            s_x\inv s_x(y) =  x=s_ys_y\inv(x),\qquad s_z(y)= s_{s_x(z)}s_x(y).
        \end{equation} In this view, $\Ris$ is the category of set-theoretic models of this algebraic theory, which is complete and cocomplete (see \cite{alg-theories}*{Thm.\ 3.4.5}). Thus, we can take quotients of risandles by \emph{congruence relations}, which are equivalence relations compatible with $s_-(-)$ and $s\inv_-(-)$; see \cite{alg-theories}*{Lem.\ 3.5.1}.

By the same coin, for all sets $X$, the \emph{free risandle} $\langle X\rangle$ generated by $X$ exists and satisfies a universal property identical to that of free groups or free quandles; see \cite{alg-theories}*{Cor.\ 3.7.8}. Furthermore, free risandles have the following canonical construction, which is a special case of \cite{alg-theories}*{Lem.\ 3.2.8}.

\begin{definition}\label{def:free}
    Given a set $X$, recursively define the \emph{universe of words} $W(X)$ to be the smallest set containing $X$ and the formal symbols $s_y(x),s_y\inv(x)$ for all $x,y \in W(X)$.
		Let $V(X)$ be the set of equivalence classes of $W(X)$ under the congruence relation generated by \eqref{eq:rlns} for all $x,y,z \in W(X)$. 
        
        In particular, for all $x\in V(X)$, the first relation in \eqref{eq:rlns} shows that the induced function $s_x\colon V(X)\to V(X)$ is bijective. So, let $s\colon V(X)\to S_{V(X)}$ be the assignment $x\mapsto s_x$. By Proposition \ref{prop:red}, $(V(X),s)$ is a risandle, so we define the \emph{free risandle generated by $X$} to be $\langle X\rangle\coloneq (V(X),s)$.
\end{definition}

\begin{example}
    If $X$ is empty, then $\langle X\rangle$ is the empty risandle.
\end{example}

\begin{example}\label{ex:free}
    Let $X=\{x\}$ be a singleton. It is easy to see from Proposition \ref{prop:id1} that the underlying set of $\langle X\rangle$ is
    \[
    \langle X\rangle = \{s_x^k(x)\mid k\in\Z\},
    \]
    which is infinite. 
    In particular, $\RMult(\langle X\rangle)\cong\Z$.
\end{example}

\section{Fundamental risandles of 3-manifolds}\label{sec:fund}
Following \cite{risandle}, we assume all $3$-manifolds $M$ to be smooth, closed, connected, and oriented. In this section, we introduce \emph{fundamental risandles} $\F(M)$, which are invariants of $3$-manifolds $M$ that categorify the risandle coloring invariants introduced in \cite{risandle}. 

\subsection{Preliminaries}
We summarize several results from \cite{risandle}. Recall from \cite{virtual} that a \emph{virtual knot diagram} is a $4$-regular planar graph whose vertices are decorated with one of three types of crossing data, namely negative and positive \emph{real crossings} (see Figure \ref{fig:crossings}) and \emph{virtual crossings}.

\begin{figure}
    \centering
    \includegraphics[alt={Two local pictures of real crossings in oriented knot diagrams. The first is a negative crossing, and the second is a positive crossing. The arcs are labeled using the right multiplication maps of the fundamental risandle.},width=0.4\linewidth]{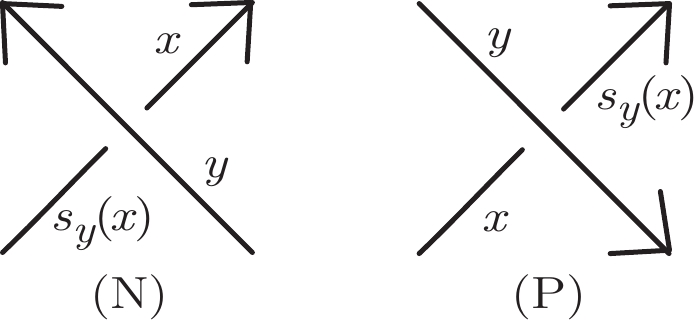}
    \caption{Relations imposed on $\F(M)$ between arcs at negative and positive real crossings.}\label{fig:crossings}
\end{figure}

Let $M$ be a $3$-manifold with a nonsingular flow $\{\phi_t\}_{t\in\mathbb{R}}$ generated by a nonsingular vector field on $M$. Then $\{\phi_t\}_{t\in\mathbb{R}}$ provides certain \emph{E-data} (also called \emph{singularity data} in \cite{flow}) that determines a \emph{flow-spine} of the pair $(M,\{\phi_t\}_{t\in\mathbb{R}})$; see \cite{flow} and cf.\ \citelist{\cite{moves}\cite{risandle}}. 

In \cite{risandle}*{Sec.\ 3}, Ishii, Nakamura, and Saito introduced a way to construct oriented virtual knot diagrams from E-data using Gauss diagrams (cf.\ \cite{virtual}) and Heegaard diagrams. For example, Figure \ref{fig:lens} depicts the virtual knot diagram of the lens space $L(n,1)$ with $n\in\Z^+$; see \cite{risandle}*{App.\ A.1} for details.

Furthermore, in \cite{risandle}*{Thm.\ 4.2}, Ishii, Nakamura, and Saito reformulate a result of Ishii \cite{moves} to show that this virtual knot diagram, considered up to an equivalence relation they call \emph{RIS-equivalence}, is invariant under orientation-preserving diffeomorphisms of $M$. Namely, two virtual knot diagrams are called RIS-equivalent if they are related by planar isotopy and a finite sequence of local moves of three types called \emph{R2-moves}, \emph{I-moves}, and \emph{S-moves}; see \cite{risandle}*{Sec.\ 4}. 

\begin{rmk}\label{rmk:reid0}
    The R2-move described in \cite{risandle} is identical to the second Reidemeister move for classical knots; this corresponds to the axiom that risandles and quandles are right quasigroups (cf.\ \citelist{\cite{joyce}\cite{matveev}}). 
    
    However, the I-move and S-move introduced in \cite{risandle} are distinct from each of the Reidemeister moves for classical knots and virtual knots (see, for example, \cite{virtual}). This corresponds to the differences between the risandle axioms and the quandle axioms; cf.\ Remark \ref{rmk:reid}. 
\end{rmk}

\subsection{Construction of $\F(M)$}

Given a $3$-manifold $M$, fix a representative $D$ of the RIS-equivalence class of virtual knot diagrams representing $M$. By the above discussion, we can refer to $D$ as ``the'' virtual knot diagram of $M$. From $D$, we construct the \emph{fundamental risandle} $\F(M)$ of $M$ in the same way that one constructs the knot quandle of a virtual link from any of its diagrams (see \cite{virtual}*{Fig.\ 11}; cf.\ \citelist{\cite{joyce}\cite{matveev}\cite{fenn}}). The only difference is that we work in $\Ris$, not the category of quandles.

Namely, if $n\in\Z^+$ and $D$ contains $n$ real crossings, then $D$ contains exactly $n$ \emph{arcs} (that is, connected components of $D$). Index the arcs by the set $X\coloneq \{x_1,\dots,x_n\}$. For each real crossing in $D$, impose relations on the free risandle $\langle X\rangle$ as shown in Figure \ref{fig:crossings}.

\begin{definition}[Cf.\ \citelist{\cite{joyce}\cite{matveev}\cite{fenn}}]
    The \emph{fundamental risandle} $\F(M)$ of $M$ is the quotient of the free risandle $\langle X\rangle$ by the congruence relation generated by the $n$ relations imposed at real crossings in~$D$.
\end{definition}

For an example of how to compute $\F(M)$ given $D$, see Section \ref{sec:lens}.

\subsection{Invariance of $\F(M)$}
For essentially the same reason that knot quandles are invariant under ambient isotopy (see \citelist{\cite{joyce}\cite{matveev}} and cf.\ \cite{fenn}), the isomorphism class of $\F(M)$ is invariant under orientation-preserving diffeomorphisms of $M$. 

\begin{thm}[Theorem \ref{thm:inv0}]\label{thm:invariant}
    If $M$ and $M'$ are 3-manifolds related by an orientation-preserving diffeomorphism, then their fundamental risandles $\F(M)$ and $\F(M')$ are isomorphic.
\end{thm}

\begin{proof}
    By \cite{risandle}*{Thm.\ 4.2}, the virtual knot diagrams of $M$ and $M'$ are RIS-equivalent. Therefore, it suffices to show that the isomorphism class of $\F(M)$ is invariant under applying R2-moves, I-moves, and S-moves to the virtual knot diagram used to construct $\F(M)$. 
    
    This verification is identical to the one for risandle coloring invariants in \cite{risandle}*{Thm.\ 5.2}. There, the authors show that invariance under the  R2-move, I-move, and S-move respectively follow from the bijectivity of $s_x$ for all $x\in X$, the risack axiom \eqref{eq:axiom}, and the risandle axioms \eqref{eq:axiom} and \eqref{eq:red}.
\end{proof}

\subsection{Categorification of coloring invariants}
In \cite{risandle}, Ishii, Nakamura, and Saito introduced invariants of $3$-manifolds $M$ based on \emph{colorings} of virtual knot diagrams $D$ of $M$ by risandles $R$. 

By the construction of $\F(M)$, a coloring of $D$ by $R$ in the sense of \cite{risandle} is equivalent to a risandle homomorphism $\F(M)\to R$. To see this, note that colorings of $D$ are precisely assignments of arcs of $D$ to elements of $R$ that preserve the relations between generators of $\F(M)$ imposed at each real crossing; see \cite{risandle}*{Fig.\ 18}. Since $\F(M)$ is generated by the set of arcs in $D$, and homomorphisms $\F(M)\to R$ are determined by where they send generators, the claim follows.

Thus, in Ishii, Nakamura, and Saito's notation, the \emph{risandle coloring number} of $M$ by $R$ equals
\begin{equation}\label{eq:cat}
    c_R(M)=|{\operatorname{Hom}_\Ris (\F(M),R)}|.
\end{equation}
In this light, fundamental risandles can be viewed as a categorification of risandle coloring invariants.
This exactly analogizes the relationship between knot quandles and quandle coloring invariants of knots; see, for example, \cite{book}*{Prop.\ 3.7}.

\section{The case of lens spaces}\label{sec:lens}
Fix a positive integer $n\in\Z^+$, and let $M$ be the lens space $L(n,1)$.
In preparation for Theorem \ref{thm:2}, we study the fundamental risandle $\F(M)$ of $M$.

\begin{figure}
    \centering
    \includegraphics[alt={A connected four-valent planar graph with n negative crossings and n arcs. The arcs are oriented and labeled x sub zero through x sub quantity n minus one.},width=0.5\linewidth]{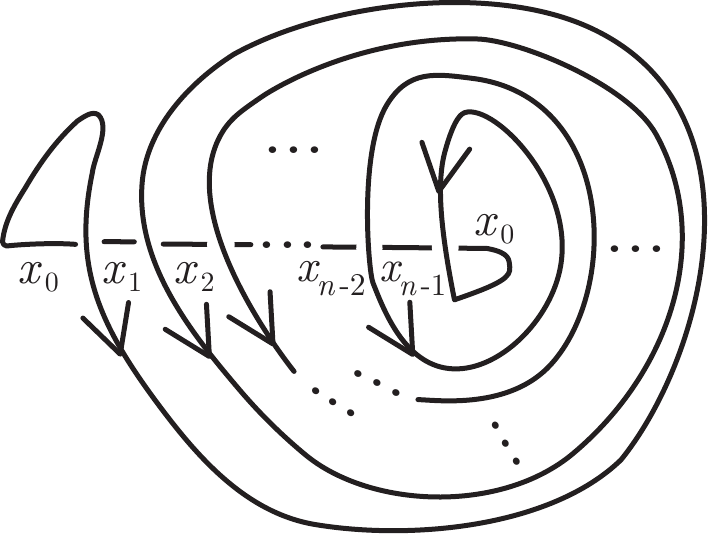}
    \caption{Virtual knot diagram $D$ of the lens space $M=L(n,1)$ with $n\in\Z^+$, as constructed in \cite{risandle}*{App.\ A.1}.}\label{fig:lens}
\end{figure}

\subsection{Computation of $\F(M)$}

Consider the virtual knot diagram $D$ of $M$ depicted in Figure \ref{fig:lens}, which Ishii, Nakamura, and Saito \cite{risandle}*{App.\ A.1} constructed using the standard Heegaard diagram of $M$. Index the arcs in $D$ by the set $X\coloneq \{x_0,\dots,x_{n-1}\}$ as depicted in Figure \ref{fig:lens}. Then $\F(M)$ is the quotient of the free risandle $\langle X\rangle$ by the congruence relation generated by the $n$ relations
    \[
    s_{x_0}(x_i)=x_{i+1},
    \]where the indices $i,i+1$ are considered modulo $n$. 

    To simplify our notation, let $x\coloneq x_0$.
    By a straightforward inductive argument, $\F(M)$ is isomorphic to the quotient of the free risandle $\langle\{x\}\rangle$ by the congruence relation generated by the relation $x=s_{x}^n(x)$. That is, \begin{equation}\label{eq:lens}
        \F(M)\cong \langle x\mid x=s_{x}^n(x)\rangle=\{s_{x}^k(x)\mid 0\leq k\leq n-1\},
    \end{equation}
    where the last equality follows from comparing $\F(M)$ with the free risandle in Example \ref{ex:free}. In particular, $|\F(M)|=n$, and $s_{x}$ is an element of order $n$ in $\RMult(\F(M))$.

\subsection{Properties} Identify the underlying set of $\F(M)$ with the third expression of \eqref{eq:lens}.

\begin{prop}\label{prop:cyclic}
    $\RMult(\F(M))$ is isomorphic to $\Z/n\Z$.
\end{prop}

\begin{proof}
    Since $|\F(M)|=n$, Lemma \ref{lem:sets} shows that $\RMult(\F(M))$ is a group of order at most $n$. Since $s_{x}$ has order $n$ in $\RMult(\F(M))$, the claim follows.
\end{proof}

\begin{cor}\label{cor:faithful}
    $\F(M)$ is faithful. 
\end{cor}

\begin{proof}
    Since $|\F(M)|=n<\infty$, the claim follows from Lemma \ref{lem:sets} and Proposition \ref{prop:cyclic}.
\end{proof}

\begin{rmk}\label{rmk:reid}
Since $\F(M)$ is finite, $\F(M)$ is not isomorphic to the free risandle with one generator; cf.\ Example \ref{ex:free}. 
    This shows that no finite sequence of planar isotopies, R2-moves, I-moves, or S-moves can recover an unknotted circle from $D$; cf.\ Remark \ref{rmk:reid0}. 
    
    By contrast, applying the first classical Reidemeister move $n$ times to $D$ does yield an unknotted circle. Accordingly, if $D$ is interpreted as a knot diagram in the usual sense (that is, up to Reidemeister moves), then the corresponding knot quandle is that of an unknot; cf.\ \citelist{\cite{joyce}\cite{matveev}\cite{fenn}}.
\end{rmk}

\section{Fundamental risandles versus fundamental groups}\label{sec:end}
\subsection{Open questions}
We pose two questions about the relationship between fundamental risandles and fundamental groups. As in Section \ref{sec:fund}, we assume all $3$-manifolds $M$ to be smooth, closed, connected, and oriented. 

\begin{prob}\label{conj1}
    Under what conditions on $M$ is there a group isomorphism
    \begin{equation}\label{eq:conj1}
        \L(\F(M))\cong \pi_1(M)?
    \end{equation}
\end{prob}

\begin{prob}\label{conj2}
    Under what conditions on $M$ is there a risandle isomorphism
    \begin{equation}\label{eq:conj2}
        \F(M)\cong \R(\pi_1(M))?
    \end{equation}
\end{prob}

Note that if $\F(M)$ is unfaithful, then the isomorphism \eqref{eq:conj2} is impossible. Conversely, Theorem \ref{thm:main} immediately yields the following.

\begin{prop}\label{prop:faithful}
    If $\F(M)$ is faithful, then the following are equivalent:
    \begin{enumerate}[({A}1)]
        \item\label{A1} The isomorphism \eqref{eq:conj1} holds.
        \item\label{A2} The isomorphism \eqref{eq:conj2} holds.
        \item\label{A3} There is a group isomorphism
        \[
        \RMult(\F(M))\cong \pi_1(M).
        \]
    \end{enumerate}
\end{prop}

\begin{cor}
    Condition \ref{A2} always implies conditions \ref{A1} and \ref{A3}.
\end{cor}

\begin{proof}
    If condition \ref{A2} holds, then $\F(M)$ is faithful, so the claim follows from Proposition \ref{prop:faithful}.
\end{proof}

\subsubsection{Discussion}
There are two motivations for Problems \ref{conj1} and \ref{conj2}. Since risandles were inspired by quandles in \cite{risandle}, we would like to find an analogue for risandles of the following result of \citelist{\cite{joyce}\cite{matveev}} (cf.\ \citelist{\cite{quasi}\cite{book}}): \emph{If $\mathcal{Q}(L)$ is the knot quandle of a link $L\subset S^3$, then there is a group isomorphism} \[\operatorname{Adconj}(\mathcal{Q}(L))\cong \pi_1(S^3\setminus L) .\]

The second motivation is more topological. 
In their paper introducing risandles, Ishii, Nakamura, and Saito \cite{risandle} used risandle coloring invariants to distinguish between all lens spaces $L(n,1)$ with $n\in\Z^+$. They similarly distinguished the $3$-sphere $S^3$ from the Poincar\'e homology $3$-sphere. 

However, these manifolds are also distinguished by their fundamental groups. This raises the question of whether risandles provide stronger or weaker invariants of $3$-manifolds than fundamental groups. If the isomorphism \eqref{eq:conj1} were true for all $3$-manifolds $M$, then fundamental risandles would be at least as strong as fundamental groups, and vice versa for \eqref{eq:conj2}. 

\subsection{An infinite class of examples}
We conclude this note by showing that infinitely many $3$-manifolds $M$ satisfy the isomorphisms \eqref{eq:conj1} and \eqref{eq:conj2}. Namely, let $n\in\Z^+$ be a positive integer, and recall that the fundamental group of the lens space $M\coloneq L(n,1)$ is \[\pi_1(M)\cong \Z/n\Z.\]

\begin{thm}[Theorem \ref{thm:3}]\label{thm:2}
    For all positive integers $n\in\Z^+$, the lens space $M=L(n,1)$ achieves all of the conditions in Proposition \ref{prop:faithful}.
\end{thm}

\begin{proof}
    By Corollary \ref{cor:faithful} and Proposition \ref{prop:faithful}, it suffices to show that $\RMult(\F(M))\cong\Z/n\Z$. But this is precisely the statement of Proposition \ref{prop:cyclic}.
\end{proof}

\begin{rmk}
    Due to Theorems \ref{thm:main} and \ref{thm:2}, a certain result of Ishii, Nakamura, and Saito \cite{risandle}*{Prop.\ 6.3} about an invariant they call \emph{effective $n$-risandle colorability} reduces to a straightforward group-theoretic statement. 
    
    Translating a definition of \cite{risandle} from risandle theory to group theory, call a group homomorphism $\phi\colon  G\to H$ \emph{effective} if, for all proper normal subgroups $N\triangleleft H$, the induced homomorphism $G\to H/N$ is nontrivial. (Equivalently, the normal closure of $\varphi(G)$ in $H$ is $H$ itself.)
    Then for all nonnegative integers $m,n\geq 0$, there exists an effective group homomorphism $\Z/m\Z\to\Z/n\Z$ if and only if $n\mid m$. This fact both recovers and provides a converse to \cite{risandle}*{Prop.\ 6.3}.
\end{rmk}

\bibliographystyle{amsplain}

\end{document}